\documentclass[a4paper,12pt]{amsart}

\usepackage[utf8]{inputenc}
\usepackage[english]{babel}
\usepackage{amsfonts}
\usepackage{amsmath}
\usepackage{mathrsfs}
\usepackage{amssymb}
\usepackage{amsthm}
\usepackage{amscd}
\usepackage{latexsym}
\usepackage{amstext}
\usepackage{amsxtra}
\usepackage[alphabetic,backrefs,lite]{amsrefs} 
\usepackage{color}
\usepackage{paralist}
\usepackage[colorinlistoftodos]{todonotes}
\usepackage[all]{xy}
\usepackage{enumerate}
\usepackage{multirow}
\usepackage{wasysym}
\usepackage{latexsym} 
\usepackage{comment}
\usepackage{colonequals}

\usepackage{geometry}
\usepackage[
        draft=false,
        colorlinks, citecolor=darkgreen,
        pdfauthor={F. F. Favale, G. P. Pirola},
        pdftitle={Variations},
        linktocpage        
]{hyperref}
\hypersetup{citecolor=blue,linktocpage}

\newtheorem{theorem}{Theorem}[section]
\newtheorem*{theorem*}{Theorem}
\newtheorem{lemma}[theorem]{Lemma}
\newtheorem{corollary}[theorem]{Corollary}
\newtheorem{proposition}[theorem]{Proposition}
\newtheorem{remark}[theorem]{Remark}
\newtheorem{definition}[theorem]{Definition}
\newtheorem{example}[theorem]{Example}

\newcommand{\nc}{\newcommand}

\nc{\cH}{{\mathcal H}}
\nc{\cA}{{\mathcal A}}
\nc{\cG}{{\mathcal G}}
\nc{\cC}{{\mathcal C}}
\nc{\cD}{{\mathcal D}}
\nc{\cO}{{\mathcal O}}
\nc{\cI}{{\mathcal I}}
\nc{\cB}{{\mathcal B}}
\nc{\cY}{{\mathcal Y}}
\nc{\cK}{{\mathcal K}} 
\nc{\cX}{{\mathcal X}}
\nc{\cS}{{\mathcal S}}
\nc{\cE}{{\mathcal E}}
\nc{\cF}{{\mathcal F}}
\nc{\cZ}{{\mathcal Z}}
\nc{\cQ}{{\mathcal Q}}
\nc{\cN}{{\mathcal N}}
\nc{\cP}{{\mathcal P}}
\nc{\cL}{{\mathcal L}}
\nc{\cM}{{\mathcal M}}
\nc{\cT}{{\mathcal T}}
\nc{\cW}{{\mathcal W}}
\nc{\cU}{{\mathcal U}}
\nc{\cJ}{{\mathcal J}}
\nc{\cV}{{\mathcal V}}
\nc{\cR}{{\mathcal R}}
\nc{\bH}{{\mathbb H}}
\nc{\bA}{{\mathbb A}}
\nc{\bG}{{\mathbb G}}
\nc{\bC}{{\mathbb C}}
\nc{\bO}{{\mathbb O}}
\nc{\bI}{{\mathbb I}}
\nc{\bB}{{\mathbb B}}
\nc{\bY}{{\mathbb Y}}
\nc{\bK}{{\mathbb K}} 
\nc{\bX}{{\mathbb X}}
\nc{\bS}{{\mathbb S}}
\nc{\bE}{{\mathbb E}}
\nc{\bF}{{\mathbb F}}
\nc{\bZ}{{\mathbb Z}}
\nc{\bQ}{{\mathbb Q}}
\nc{\bN}{{\mathbb N}}
\nc{\bP}{{\mathbb P}}
\nc{\bL}{{\mathbb L}}
\nc{\bM}{{\mathbb M}}
\nc{\bT}{{\mathbb T}}
\nc{\bW}{{\mathbb W}}
\nc{\bU}{{\mathbb U}}
\nc{\bD}{{\mathbb D}}
\nc{\bJ}{{\mathbb J}}
\nc{\bV}{{\mathbb V}}
\nc{\bR}{{\mathbb R}}

\nc{\OO}{\mathcal{O}}
\nc{\PP}{\mathbb{P}}

\DeclareMathOperator{\id}{id}
\DeclareMathOperator{\im}{Im}

\DeclareMathOperator{\Hom}{Hom}

\DeclareMathOperator{\Img}{Im}

\DeclareMathOperator{\Ker}{Ker}

\DeclareMathOperator{\Sym}{Sym}
\DeclareMathOperator{\Rk}{Rk}
\DeclareMathOperator{\rk}{Rk}

\DeclareMathOperator{\Edim}{Edim}

\newcommand{\KS}{KS}

\nc{\fA}{{\mathfrak{A}}}
\nc{\fB}{{\mathfrak{B}}}
\nc{\fC}{{\mathfrak{C}}}
\nc{\fD}{{\mathfrak{D}}}
\nc{\fE}{{\mathfrak{E}}}
\nc{\fF}{{\mathfrak{F}}}

\nc{\p}{\partial}
\nc{\ph}{\hat{\partial}}

\nc{\war}{{\color{red} CHECK}}

\begin{document}

\title[Infinitesimal variation functions for families of smooth varieties]{Infinitesimal variation functions for\\ families of smooth varieties}

\date{\today}

\author{Filippo Francesco Favale}
\address{Dipartimento di Matematica e Applicazioni,
	Universit\`a degli Studi di Milano-Bicocca,
	Via Roberto Cozzi, 55,
	I-20125 Milano, Italy}
\email{filippo.favale@unimib.it}

\author{Gian Pietro Pirola}
\address{Dipartimento di Matematica,
	Universit\`a degli Studi di Pavia,
	Via Ferrata, 5
	I-27100 Pavia, Italy}
\email{gianpietro.pirola@unipv.it}

\date{\today}
\thanks{
ACKNOWLEDGEMENTS: \\
The authors are thankful to Enrico Schlesinger for helpful discussion on the topics of the paper. The authors are partially supported by INdAM - GNSAGA. The first named author is partially supported by ``2017-ATE-0253'' (Dipartimento di Matematica e Applicazioni - Universit\`a degli Studi di Milano-Bicocca). The second named author is partially supported by PRIN 2017 \emph{``Moduli spaces and Lie theory''} and by (MIUR): Dipartimenti di Eccellenza Program (2018-2022) - Dept. of Math. Univ. of Pavia. 
}


\maketitle

\begin{abstract}
In this paper we introduce some {\it variation functions} associated to the rank of the Infinitesimal Variations of Hodge Structure for a family of smooth projective complex curves. We give some bounds and inequalities and, in particular, we prove that if $X$ is a smooth plane curve $X$, then there exists a first order deformation $\xi\in H^1(T_X)$ which deforms $X$ as plane curve, such that $\xi\cdot:H^0(\omega_X)\to H^1(\cO_{X})$ is an isomorphism. We also generalize the notions of variation functions to the higher dimensional case and we analyze the link between IVHS and the Weak and Strong Lefschetz properties of the Jacobian ring of a smooth hypersurface.
\end{abstract}


\section*{Introduction}

The period mapping associates  a variation of Hodge structure (VHS) (\cite{Gri_I,Gri_II}) to a family of complex algebraic varieties. This operation can be seen as a partially linearization functor. 
The periods of an algebraic variety are objects of transcendental nature, whereas its differential is of algebraic nature. This leads to the infinitesimal variation of Hodge structure (IVHS) (\cite{IVHS_I,IVHS_II,IVHS_III}), which is a complete linearization functor.

The IVHS has been proved to be an important tool to tackle several interesting problems and in dealing with a lot of applications. For several families of varieties, e.g. many smooth projective hypersurfaces, IVHS was proved powerful enough to give a complete reconstruction of the algebraic variety, a Torelli-type theorem (see \cite{Don,Voi,GrT}, for example).

Recently the IVHS of families curves has been studied in connection with the Fujita decomposition (see \cite{GT}), with the Xiao conjecture (see \cite{BGN,FNP,BPZ}) and with holomorphic forms on some moduli spaces (see \cite{FPT}). In these articles, very roughly, the authors tried to use the rank of differential of the period mapping in a systematic way. 
Accordingly, in this article we introduce some functions, associated to families of smooth curves (which we generalise then to the case of families of arbitrary smooth varieties) that give a quantitative measure of the variation of the family. We can call them {\it variation functions}. If the simplest one is the dimension of the image of the modular map (this has been studied in many papers - see \cite[Section 2.3]{GP}, for example - and again in the recent article \cite{DH}), in this paper we consider some variation functions which take into account the IVHS.

To be more concrete, consider, if $\pi:\cX\to B$ is a family of algebraic smooth curves and $0\in B$ let $X=\pi^{-1}(0)$ be an element of this family. The differential $d_0\cP$ of the period map $\cP$ in $0$ is the composition of two maps. The first one is the Kodaira-Spencer map $\KS:T_{B,0}\to H^1(T_X)$  and second is the IVHS map
$$\varphi:H^1(T_X)\to \Hom(H^0(\omega_X),H^1(\cO_X)),$$
which sends $\eta\in H^1(T_X)$ to the cup product by $\eta$ and it is the dual map of the multiplication map between canonical sections $H^0(\omega_X)^{\otimes 2}\to H^0(\omega_X^{\otimes 2})$. 
Inspired by the importance of the study of the ranks of these maps in many of the papers cited above,  we introduce two functions $d_M$ and $d_m$ which compute, respectively the maximum rank and the minimal rank of the cup product by elements in $KS(T_{B,0})\setminus\{0\}$ (for detail see Definition \ref{DEF:MAXMINVAR}). Then we use $d_M$ and $d_m$ to define the variation functions $\delta_M,\delta_m,\delta_M'$ and $\delta_m'$ associated to $\pi:\cX\to B$, which take into account all the elements of the family (see Definition \ref{DEF:MAXMINVARFAMIIES}). We say that a family $\pi$ has {\bf $I$-maximal variation} if $\delta_M'(\pi)=\min_{b\in B} d_M(\KS(T_{B,b}))$ is equal to the genus of the curves in this family. In particular, a family $\pi:\cX\to B$ has $I$-maximal variation if for all $b\in B$ there exists $\xi\in \KS(T_{B,b})$ such that $\xi\cdot$ is an isomorphism.
\vspace{2mm}

The structure of the paper is as follows: first of all, in Section \ref{SEC:VAR}, we introduce the variation functions and prove some general bounds on them. We also propose one of the possible generalizations to the case of (families of) smooth varieties of dimension $n$: the Yukawa-Coupling map (see \cite[Construction 1, pag. 53]{GrT})
$$\varphi:\Sym^n(H^1(T_X))\to \Hom(H^0(\omega_X)\to H^n(\cO_X))$$
sending $\bigotimes_{i=1}^{n}\xi_i$ to the cup product by $\Pi_{i=1}^{n}\xi_i$. In Section \ref{SEC:NOTATIONS}, we introduce some notations and prove some technical results that we need for proving our main result that is presented  
in Section \ref{SEC:VARPLA}, where we concentrate ourselves on the case of families of smooth plane curves and prove: 

\begin{theorem*} [Theorem \ref{THM:PLANARCURVESMAXVAR}]
Smooth plane curves of degree $d$ have $I$-maximal variations as smooth curves, i.e. if $\pi$ is the family of smooth curves induced by $|\cO_{\bP^2}(d)|$, we have  that $\delta_M'(\pi)=(d-1)(d-2)/2$.
\end{theorem*}

The techniques involved are a mix of results on Jacobian rings, geometric constructions and  the Castelnuovo Uniform Position Theorem (see Proposition \ref{PROP:JETSINJDEC}), used  to bound the dimension of certain decomposable elements contained in the Jacobian ideal, which could be interesting on its own.
\vspace{2mm}

We notice that our main Theorem is complementary to a result in \cite{FNP}, where the minimum of the variation function for a family of planar curves of degree $d\geq5$ was computed (see Proposition \ref{PROP:MINIMALVAR}).
Finally, in Section \ref{SEC:YC}, we analyze the higher dimentional case: we prove that the general hypersurface of degree $d$ in $\bP^n$ has $I$-maximal variation as hypersurface of $\bP^n$ (see Proposition \ref{PROP:MAXVARHYPERSUP}). 
\vspace{2mm}

When $X$ is a smooth hypersurface in $\bP^n$, the cohomology of $X$ is codified in its jacobian ring $R$, which is an Artinian graded standard algebra. For such algebras there are two properties that have been studied a lot and whose interest is still very appealing nowadays: the Weak and Strong Lefschetz Properties (see Section \ref{SEC:NOTATIONS} for a description of these properties and some related results). 
There is a clear link between SLP and our variation functions: indeed, it is easy to see that if the jacobian ring of an hypersurface $X$ of $\bP^n$ has the SLP, then $X$, as hypersurface in $\bP^n$, has $I$-maximal variation (see Lemma \ref{LEM:SLPMAXVAR}).
\vspace{2mm}

Among the various results about these topics we would like to highlight a result which implies that Fermat hypersurfaces of degree $d$ in $\bP^n$ have jacobian rings which satisfy the SLP. These results can be used to prove also our Proposition \ref{PROP:MAXVARHYPERSUP}. Nevertheless, we decided to leave our proof since in the literature there are several and very different proofs \footnote{The first known proof (\cite{BE}) is of combinatorial nature, whereas Stanley (\cite{Sta}) proposed a geometric, concise and elegant proof which works when $\bK=\bC$. There are also algebraic proofs (like \cite{RRR,Wat} and the recent \cite{Lin}).} of the above result (see \ref{REM:SUNTOLEF} for the precise statement). Our proof relies on a induction argument performed by using partial derivatives and the Euler identity.
\vspace{2mm}

Finally, we would like to remark that our main result, i.e. the fact that planar curves have maximal variation (as planar curves), is not a consequence of any known result regarding either the WLP or the SLP. The SLP is conjectured to hold (see Remark \ref{REM:SUNTOLEF}) for standard graded Artinian algebras of codimension $3$. Hence the main result of this article (Theorem \ref{THM:PLANARCURVESMAXVAR}), gives an evidence for the validity of this important conjecture.


\section{Variation functions}
\label{SEC:VAR}
Let $X$ be a smooth projective curve over $\bC$. Let 
$$\varphi: H^1(T_X)\to \Hom(H^0(\omega_X),H^1(\cO_X))$$ 
be the map induced by the infinitesimal variation of the periods. It is the dual map of the multiplication map $H^0(\omega_C)^{\otimes 2}\to H^0(\omega_C^{\otimes 2})$. We give two numbers to measure the variation on subspaces of $H^1(T_X)$.

\begin{definition} 
\label{DEF:MAXMINVAR}
Let $U\neq \{0\}$ be a vector subspace of $H^1(T_X)$. We set 
$$d_M(U)= \max_{\xi\in U} \dim\varphi(\xi)(H^0(\omega_X))=\max_{\xi\in U} \Rk(\varphi(\xi))$$ and $$d_m(U)= \min_{\xi\in U, u\neq 0} \dim\varphi(\xi)(H^0(\omega_X))=\min_{\xi\in U, u\neq 0} \Rk(\varphi(\xi))$$
and call them variations of $U$. We say that $U$ has {\bf $I$-maximal variation} if $\delta_M(U)=g$ that is there is $\xi\in U$ such that $\varphi(\xi)$ is an isomorphism.
\end{definition}

\begin{remark}
\label{REM:FULLH1}
If $X$ is a smooth projective curve of genus $g\geq 3$ we have $d_m(H^1(T_X))\in \{0,1\}$ with $d_m(H^1(T_X))=1$ if and only if $X$ is not hyperelliptic. If $g\geq 2$ we have $d_m(H^1(T_X))=1$.
\end{remark}

Using a result in \cite[Lemma 2.3]{LP} one can prove that for the very general non-hyperelliptic curve, $H^1(T_X)$ has maximal variation. The following proposition strengthen this result by extending it to any smooth curve.

\begin{lemma}
\label{LEM:FULLMAXVAR}
Let $X$ be a smooth curve of genus $g$. Then $d_M(H^1(T_X))=g$, i.e. $H^1(T_X)$ has $I$-maximal variation.
\end{lemma}

\begin{proof}
The aim is to prove that there exists $\xi\in H^1(T_X)$ such that the cup product $\xi\cdot : H^0(\omega_X)\to H^1(\cO_X)$ is an isomorphism or, equivalently, injective. Assume, by contradiction, that this is not the case so $\xi\cdot$ is never injective. Take $\xi$ such that $\xi\cdot$ is of maximal rank and let $\alpha\in H^0(\omega_X)\setminus\{0\}$ be such that $\xi\cdot \alpha=0$. 
We will prove that $\xi'\alpha^2=0$ for all $\xi'\in H^1(T_X)$ and then see that this yields a contradiction. 
\vspace{2mm}

Let $\xi'\in H^1(T_X)$. If $\xi'\alpha=0$ then we have also $\xi'\alpha^2=0$ and there is nothing to prove. If, instead, $\xi'\alpha\neq 0$ consider $h_t=\xi+t\xi'$. Since $\xi$ is such that $\varphi(\xi)$ has maximal rank and is not injective we have also that $\varphi(\xi)$ has kernel of minimal dimension.
In particular, $h_t$ has non trivial kernel for all $t$. Let $\gamma$ such that $\gamma(t)\in \Ker(h_t)$ for all $t$ and $\gamma(0)=\alpha$. By considering the expansion of $\gamma(t)$, we can assume that the first term of the expansion of $\gamma$ which is not $0$ has order $k\geq 1$. Then
$\gamma(t)=\alpha+\alpha' t^k+o(k+1)$ and we have
$$0\equiv h_t(\gamma(t))=\xi(\alpha+\alpha't^k+o(k+1))+t\xi'(\alpha+\alpha't^k+o(k+1)) = t\xi'\alpha+t^k\xi\alpha'+o(k+1)$$
which is impossible unless $k=1$ since $\xi'\alpha\neq 0$. Then $k=1$ and we get $\xi'\alpha+\xi\alpha'=0$.
Then, by multiplying by $\alpha$ we have 
$$0=\xi'\alpha^2+\xi\alpha'\alpha=\xi'\alpha^2$$
as claimed. 
\vspace{2mm}

Since $\alpha^2\in H^0(\omega_X^2)$ and $\xi'\in H^1(T_X)=H^0(\omega_X^2)^*$, having $\xi'\alpha^2=0$ for all $\xi'$ is possible only if $\alpha^2=0$. But this is impossible since we have assumed $\alpha\neq 0$. Therefore, there exists $\xi\in H^1(T_X)$ such that $\xi\cdot:H^0(\omega_X)\to H^1(\cO_X)$ is an isomorphism.
\end{proof}

We now specialize the notion of variations just introduced to families of curves. If $\pi:\cX\to B$ is a family of smooth curves of genus $g$ over a smooth base $B$, then for any $b\in B$ we have the Kodaira-Spencer map
$\KS:T_{B,b}\to H^1(T_{X_b})$ where $X_b=\pi^{-1}(b)$.
\vspace{2mm}

We define now the variation function discussed informally in the introduction.

\begin{definition} 
\label{DEF:MAXMINVARFAMIIES}
We set 
$$\delta_M(\pi)=\max_{b\in B} d_M (\KS(T_{B,b}))\qquad\qquad  \delta_M'(\pi)=\min_{b\in B} d_M (\KS(T_{B,b}))$$
$$\delta_m(\pi)=\max_{b\in B} d_m (\KS(T_{B,b}))\qquad\qquad  \delta_m'(\pi)=\min_{b\in B} d_m (\KS(T_{B,b}))$$
and call them variations functions related to $\pi$. We will say that the family has {\bf $I$-maximal variation} if $\delta_M(\pi)=g.$
Given a smooth projective surface $S$ and $X$ is a smooth curve in $|L|$ where $L$ is a line bundle on $S$, we define $\delta_M(L),\delta_M'(L),\delta_m(L)$ and $\delta_m'(L)$ when we consider the family of smooth curves defined by the section of $L$.
\end{definition}

We have the following corollary of Lemma \ref{LEM:FULLMAXVAR}.

\begin{corollary}
For a family $\pi:\cX\to B$ for which the moduli map is dominant (e.g. for a versal family), we have $\delta_M(\pi)=\delta_M'(\pi)=g$ by Lemma \ref{LEM:FULLMAXVAR}. We have $\delta_m(\pi)=\delta_m'(\pi)=1$ with possible exceptions when the family contains an hyperelliptic curves.
\end{corollary}

\begin{remark}
\label{REM:KS}
If $S$ is a surface, $L$ is a line bundle and $f\in H^0(L)\setminus\{0\}$ is smooth (i.e. $f$ is such that $X=Z(f)$ is smooth), the family induced by $|L|$ has base $B$ which is an open in $\bP(H^0(\cO_S(X)))$ and $T_{B,X}\simeq H^0(\cO_S(X))/\langle f\rangle$. If we consider the restriction $\rho:H^0(\cO_S(X))\to H^0(\cO_X(X))=H^0(N_{X/S})$ induced by $\cO_S(X)\to \cO_X(X)=N_{X/S}$ and the coboundary map $\partial:H^0(N_{X/S})\to H^1(T_X)$ of the tangent sequence, the Kodaira-Spencer map is induced by the composition $\partial\circ \rho=\KS'$. 
$$\xymatrix{
0 \ar[r] &
  H^0(\cO_S) \ar[r]^-{\cdot f} &
  H^0(\cO_S(X)) \ar[r]^{\rho} \ar[rd]_{\KS'}&
  H^0(N_{X/S}) \ar[r] \ar[d]^{\partial} &
  H^1(\cO_S)\\
 & & & H^1(T_X) 
}$$
We will say that $X$ has {\bf $I$-maximal variation in $S$} if the subspace $U=\KS'(H^0(\cO_S(X)))=\Img(\KS)$ has maximal variation for $X$.
\end{remark}

\begin{proposition}
Let $S$ be a surface and $L$ an ample line bundle with a smooth section. Then $\delta_M(L)\leq g-q(S)$ where $g$ is the genus of the general curve in $|L|$. 
\end{proposition}

\begin{proof}
Let $Z(f)=X\in |L|$ be a smooth element and fix the notation as in Remark \ref{REM:KS}. By Kodaira vanishing, we have the inclusions $H^0(\Omega_S^1)\hookrightarrow H^0(\Omega_S^1|_X)$ and  $j:H^0(\Omega_S^1|_X)\to H^0(\omega_X)$ (the latter is induced by the cotangent sequence). The composition $\gamma$ of these maps is simply the restrictions $H^0(\Omega_S^1)\to H^0(\Omega_X^1)$ of $1$-forms on $S$ to $X$. Hence, varying $X\in |L|$, we have that $H^0(\omega_X)$ contains a constant part which comes from $H^0(\Omega_S^1)$. We want to show that these $1$-forms annihilate the elements in the image of the Kodaira-Spencer map. Consider $\alpha\in H^0(\Omega_S^1)$ and $\eta\in H^0(\cO_S(X))$. The cup product $\KS'(\eta)\cdot \gamma(\alpha)\in H^1(\cO_X)$ can be computed as follows. First of all, one takes the cup product $\omega\cdot \eta\in H^0(\Omega_S^1(X))$ and restricts it to $X$. Then we have a map sending $(\omega\cdot \eta)|_X$ to $H^1(\cO_X)$ which factors through the exact sequence
$$0\to H^0(\cO_X)\to H^0(\Omega_S^1(X)|_X)\to H^0(\omega_X(X))\to H^1(\cO_X)\to \cdots.$$
Hence, the image of $(\omega\cdot \eta)|_X$ in $H^1(\cO_X)$ is $0$. Then, for all $\eta\in \Img(\KS)$ we have $\gamma(H^0(\Omega_S^1))\subset \Ker(\eta\cdot: H^0(\omega_X)\to H^1(\cO_X))$ as claimed.
\end{proof}

\begin{remark}
Let $f:S\to B$ be a non-isotrivial fibration with $S$ and $B$ a smooth (of dimension $2$ and $1$ respectively) and let $F_b=f^{-1}(b)$ be a smooth fiber with genus $g\geq 2$. If $q_f$ is the relative irregularity of $f$ and $\xi_b\in H^1(T_{F_b})$ is the first order deformation of $F_b$ induced by $f$. Then 
$$q_f\leq g-\rk(\xi\cdot: H^0(\omega_{F_b})\to H^1(\cO_{F_b}))$$
(see, for instance, in \cite{BGN}). Notice that the above inequality can be strict (see, for example, \cite{Fla,GT}).
If $f'=f|_U:U\to B_o$ where $U=f^{-1}(B^o)$ and $B^o\subset B$ is the open which parametrizes smooth fibers, the above inequality implies $q_f\leq g-d_M'(f')=g-d_m'(f')$.
\end{remark}

Now we will concentrate on families of plane curves. First of all we will reinterpret a result of the authors (see \cite{FNP}) in the framework of the variation functions.

\begin{proposition}
\label{PROP:MINIMALVAR}
For all $d\geq 5$ we have $\delta_m'(\cO_{\bP^2}(d))=d-3$ and $\delta_m(\cO_{\bP^2}(d))\geq d-3$.
\end{proposition}

\begin{proof}
Consider a smooth curve of degree $d\geq 5$ on $S=\bP^2$ and let $\pi:\cC\to B$ be the family of smooth plane curves of degree $d$. 
In \cite[Theorem 1.3]{FNP}, it is proved that the rank of the cup product by $\xi\in \Img(\KS)$ is at least $d-3$, unless $\xi=0$. Then $\delta_m(\cO_{\bP^2}(d))\geq d-3$ and $\delta_m'(\cO_{\bP^2}(d))=d-3$ since for Fermat curve of degree $d$ one can easily write an infinitesimal deformation with rank $d-3$.
\end{proof}

One of the mail result of this paper is that the family of curves in $\bP^2$ has $I$-maximal variation (see Theorem \ref{THM:PLANARCURVESMAXVAR}). This will be stated and proved in Section \ref{SEC:VARPLA}. 
\vspace{2mm}

We conclude this section by giving one of the possible generalization (perhaps, the more extreme) of our definition of variations to the case of higher dimensional varieties.

\begin{definition} 
Let $X$ be a smooth complete projective variety of dimension $n$. Let 
$$\varphi:Sym^n(H^1(T_X))\to \Hom(H^0(\omega_X)\to H^n(\cO_X))$$
the Yukawa coupling mapping (\cite[Construction 1, pag. 53]{GrT}), i.e. the map sending $\bigotimes_{i=1}^{n}\xi_i$ to the cup product by $\Pi_{i=1}^{n}\xi_i$.
For any $\xi\in H^1(T_X)$ write
$\varphi(\xi^{\otimes n})=\xi^n\cdot: H^0(\omega_X)\to H^n(\cO_X)$. We say then that $\xi$ has {\bf $I$-maximal variation} if $\xi^n$ induces an isomorphism.
We set, for $U\subset H^1(T_X)$
$$\delta_M(U)= \max_{\xi\in U} \dim\xi^n\cdot(H^0(\omega_X))\qquad \mbox{ and }\qquad \delta_m(U)= \min_{\xi\in U, \xi\neq 0} \dim\xi^n\cdot(H^0(\omega_X)).$$
Accordingly, we can define the numbers $\delta_M,\delta_M',\delta_m$ and $\delta_m'$ associated to a families of variety and of line bundles.
\end{definition}


\section{Gorenstein Rings an Lefschetz properties}
\label{SEC:NOTATIONS}

In this section we fix some notations, recall some well known facts and prove two lemmas. Good references for Jacobian rings, their relation to cohomology of hypersurfaces and the IVHS of the latter are the books \cite{textVoi1,textVoi2,bookLef,BH} or the original works \cite{Gri_I,Gri_II,IVHS_I,IVHS_II,IVHS_III}.
\vspace{2mm}

\begin{definition}
\label{DEF:GOR}
Let $\bK$ be a field and consider a standard\footnote{This means that $R$ is generated, as $\bK$-algebra, by $R^1$, i.e. the vector space of elements of $R$ of degree $1$.} graded Artinian $\bK$-algebra $R=\bigoplus R^{s}$ of finite dimension. Then $R$ is a {\bf Gorenstein ring} if the following hold:
\begin{itemize}
    \item $R^s=0$ if $s>N$ or $s<0$ and $R^0\simeq R^N\simeq \bK$ as vector spaces;
    \item the multiplication map $R^{a}\times R^{N-a}\to R^N$ is a perfect pairing for all $0\leq a\leq N$.
\end{itemize}
The graded piece of degree $N$, namely $R^N$, is called {\bf socle} of $R$ whereas we refer to the second property as ``Gorenstein duality'' (since it induces isomorphisms $R^{N-k}\simeq (R^{k})^*$) for brevity.
\end{definition}

Notice that, although the definition of Gorenstein ring is more general, in this paper we will only encounter Gorenstein rings which are like in definition \ref{DEF:GOR}. Hence, we will simply say that a ring is Gorenstein for brevity. The main examples that we will use are Gorenstein rings as the following.  

\begin{example}
\label{EX:hypers}
If $X=V(F)$ is a smooth hypersurface in $\bP^n$ of degree $d$ and $J=(F_{x_0},\dots,F_{x_n})$ is the Jacobian ideal associated to $F$, we have that $R=\bK[x_0,\dots,x_n]/J$ is a Gorenstein ring with socle in degree $N=(n+1)(d-2)$. It is called the {\bf Jacobian ring} associated to $X$.
\end{example}

Let $R=\oplus R^s$ be a Gorenstein ring with socle in degree $N$. Take $\alpha\in R^e\setminus\{0\}$ and consider the multiplication map $\mu:R \stackrel{\cdot\alpha}\to R$. 
Since $\alpha\in R^e$, we have that 
$$\mu_s(\alpha):R^s \stackrel{\cdot\alpha}\to R^{s+e}$$
is a graded morphism and we can set $K_{s}(\alpha)=\ker(\alpha\cdot :R^s\to R^{s+e})$. 
The quotient ring $R_{\alpha}=R/(0:\alpha)$ has then a natural graded structure with $(R_{\alpha})^s=R^s/K_{s}(\alpha)$. Denote by $r_j$ and $k_j(\alpha)$ (or simply $k_j$, if no confusion arises) the dimension of $R^j$ and of $K_{\alpha}^j$ respectively.

\begin{lemma} 
\label{gring}
Let $R=\oplus R^s$ be a Gorenstein ring with socle in degree $N$. For all $\alpha\in R^e\setminus\{0\}$ we have that $R_{\alpha}=R/(0:\alpha)$ is a Gorenstein ring with socle in degree $N_{\alpha}=N-e$. Moreover
$$r_{N-e-s}-k_{N-e-s}=r_s-k_s$$
for all $s$ with $0\leq s\leq N-e$.
\end{lemma}

\begin{proof}
First of all, notice that $R_{\alpha}^{N-e}$ is one dimensional. Indeed, the map $\alpha\cdot R^{N-e}\to R^N$ is surjective (by Gorenstein duality in $R$, as $\alpha\neq 0$). Then $k_{N-e}=r_{N-e}-1$ and $\dim(R_{\alpha}^{N-e})=r_{N-e}-k_{N-e}=1$ as claimed. Now we show that Gorenstein duality holds for $R_{\alpha}$. 
Let $[\beta]\in R_{\alpha}^{s}\setminus\{[0]\}$ with $0\leq s\leq N_{\alpha}=N-e$. Since $[\beta]\neq [0]$ we have that $\alpha\beta\neq 0$ in $R^{s+e}$. 
Then, by Gorenstein duality on $R$, there exists $\gamma\in R^{N-s-e}$ such that $\alpha\beta\cdot \gamma\neq 0$ in $R^{N}$. In particular, $(\beta\gamma)\cdot \alpha \neq 0$ so $\beta\gamma\not \in K_{N-e}(\alpha)$, i.e. $[\beta\gamma]\neq 0$ in $R_{\alpha}^{N_{\alpha}}$ and Gorenstein duality holds for $R_{\alpha}$ as claimed.
The relation among $r_j$ and $k_j$ simply follows from $\dim(R_{\alpha}^s)=\dim(R_{\alpha}^{N_{\alpha}-s})$.
\end{proof}

Using Gorenstein duality instead of the duality between $H^1(T_X)$ and $H^0(\omega_X^{2})$ we are able to prove the following lemma which is an analogue of Lemma \ref{LEM:FULLMAXVAR}.

\begin{lemma}
\label{LEM:kercokerLAST}
Let $R=\bigoplus_{k=0}^{N}$ be a Gorenstein ring with socle in degree $N$ and fix $0\leq d,e,\leq N$ with $d+2e\leq N$. 
Then, for all $\eta\in R^d$ with $\eta\cdot:R^{e}\to R^{d+e}$ of maximal rank and for all $\alpha\in K_s(\eta)$ we have $\alpha^2=0$.
\end{lemma}

\begin{proof}
The proof is similar to the one of Lemma \ref{LEM:FULLMAXVAR}. Let $\eta\in R^{d}$ of maximal rank. If $\eta$ is injective the thesis holds trivially so we can assume that $\ker(\mu_e(\eta))=K_{e}(\eta)\neq \{0\}$ and $d\neq 0,N$. Let $\alpha\in K_{e}(\eta)\setminus\{0\}$ and consider $\eta'\in R^d$. If $\eta'\alpha=0$ then we have also $\eta'\alpha^2=0$. We want to show that $\eta'\alpha^2=0$ also if $\eta'\alpha\neq 0$. As in Lemma \ref{LEM:FULLMAXVAR}, one has can find a curve $\gamma$ such that $\gamma(0)=\alpha$ and $\gamma(t)\in K_e((\eta+t\eta')\cdot)\setminus\{0\}$. Then we have
$\eta \alpha'+\eta' \alpha=0$ and then $\eta'\alpha^2=0$. Then, by multiplying by $\alpha$ we would have 
$$0=\eta'\alpha^2+\eta\alpha'\alpha=\eta'\alpha^2$$
since $\eta\alpha=0$. Hence, as claimed, $\eta'\alpha^2=0$ for all $\eta'\in R^d$ and this proves that $\alpha^2=0$ by the non-degeneracy of the product $R^d\times R^{2e}\to R^{2d+e}$ (which follows by the Gorenstein duality $R^d\times R^{N-d}\to R^{N}$).
\end{proof}

We conclude the section by recalling the definition of weak and strong Lefschetz property. The reader can refer to \cite{bookLef} for a comprehensive text about the Lefschetz properties.

\begin{definition}
Let $\bK$ be a field and consider a standard Artinian graded $\bK$-algebra $R=\bigoplus_{s=0}^{N} R^k$. We say that $R$ satisfy the 
\begin{itemize}
    \item {\bf Weak Lefschetz Property} (WLP) if there exists $L\in R^1$ such that $L\cdot :R^k\to R^{k+1}$ ha maximal rank for all $k$;
    \item {\bf Strong Lefschetz Property} (SLP) if there exists $L\in R^1$ such that $L^d\cdot :R^k\to R^{k+d}$ has maximal rank for all $k,d$.
\end{itemize}
The {\bf codimension} of $R$ is, by definition, the number of generators of $R^1$ as $\bK$-vector space.
\end{definition}

We summarize here some results and conjecture relevant with respect to the topics of our article.

\begin{remark}
\label{REM:SUNTOLEF}
Le $R$ be a standard Artinian graded $\bK$-algebra. Then:
\begin{itemize}
    \item if $R$ has codimension $2$ or less, then $R$ satisfies the SLP;
    \item if $R$ has codimension $3$ and is a complete intersection ring, then $R$ satisfy the WLP;
    \item No examples of complete intersection ring with codimension $3$ which does not satisfy either WLP or SLP are known;
    \item if $R$ is a complete intersection ring, then it is conjectured that $R$ satisfy also SLP;
    \item Theorem: let $\bK$ be a field of characteristic zero. Then $R=\bK[x_0,\dots,x_n]/(x_0^{a_0},\dots, x_n^{a_n})$ satisfy the SLP. 
\end{itemize}
As specified in the introduction, several proof of the latter statement exist in literature. Up to our knowledge, the first proof of the Theorem can be found in \cite{BE}.
\end{remark}

\begin{remark}
\label{REM:SLP}
If $X$ is a smooth hypersurface in $\bP^n$ and $R$ is its Jacobian ring, then $R$ is a Gorenstein ring with socle in degree $N=(d-1)(n+1)$ of codimension $n+1$ which is also a complete intersection ring (since it is a regular ring). In particular, it is conjectured (see Remark \ref{REM:SUNTOLEF}) that $R$ satisfy SLP. 
\end{remark}

The following lemma state the link between the variation functions for hypersurfaces in $\bP^n$ and the SLP.

\begin{lemma}
\label{LEM:SLPMAXVAR}
Let $X$ be a smooth hypersurface of $\bP^n$ and let $R$ be its Jacobian ring. If $R$ has the SLP then $X$ has $I$-maximal variation as hypersurface of $\bP^n$.
\end{lemma}

\begin{proof}
If $R$ satisfy the SLP then there exists $L\in R^1=H^0(\cO_{\bP^n}(1))$ such that $L^{d(n+1)}\cdot :R^{d-n-1}\to R^{N-d-1}$ is an isomorphism. Since $R^{d-n-1}=H^0(\cO_{\bP^n}(d-n-1))\simeq H^0(\omega_X)$ by adjunction and by the results of Griffiths, Green and Donagi about the cohomology of hypersurfaces of $\bP^n$ and the IVHS of the latter, we have that the above multiplication map is simply the map $\varphi(L^{\otimes d(n-1)})=L^{\otimes d(n-1)}\cdot :H^0(\omega_X)\to H^{n-1}(\cO_X)$. Then, we have that $X$, as hypersurface, has $I$-maximal variation.
\end{proof}

In particular, by Remark \ref{REM:SLP}, it is conjectured that hypersurfaces in $\bP^n$ should have $I$-maximal variation as hypersurfaces.
The main result of this article prove this conjecture for the case $n=2$, i.e. for plane curves (see Theorem \ref{THM:PLANARCURVESMAXVAR}).


\section{Plane curves}
\label{SEC:VARPLA}

In this section we prove the main theorem of the article:

\begin{theorem}
\label{THM:PLANARCURVESMAXVAR}
Smooth planar curves of degree $d\geq 3$ have $I$-maximal variation as planar curves. More precisely, $\delta_M(\cO_{\PP^{2}}(d))=\delta_M'(\cO_{\PP^{2}}(d))=\frac{(d-1)(d-2)}{2}$ for $d\geq 3$.
\end{theorem}

Before the proof, we introduce some notation and prove two technical propositions. Let $X=V(F)$ a smooth plane curve of degree $d$ and let $g=(d-1)(d-2)/2$ its genus. We will denote by $S=\bigoplus_{k\geq 0}S^k$ where $S^k=H^0(\cO_{\bP^2}(k))$ and by $J$ the Jacobian ideal associated to $X$, i.e. the ideal spanned by partial derivatives of $F$: $J=(F_{x_0},F_{x_1},F_{x_2})$. As recalled in Example \ref{EX:hypers}, the associated Jacobian ring $R=S/J$ is a Gorenstein ring with socle in degree $N=3d-6$. The graded pieces of $R$ are $R^k=S^k/J^k$ where $J^k=S^k\cap J$.
\vspace{2mm}

If $\alpha\in R^e$, we recall that we have set $K_{m}(\alpha)=\ker(\alpha\cdot R^{m}\to R^{m+e})$ and $k_{m}=k_m(\alpha)=\dim(K_{m}(\alpha))$. Since we defined $r_k$ to be equal to $\dim(R^k)$, it is convenient to set $s_k=\dim(S^k)=h^0(\cO_{\bP^2}(k))$.
\vspace{2mm}

If $p\geq 0$, according to the parity of $p$, we define $\tilde{s}$ to be one of the following Segre morphisms: 
\begin{equation}
\label{EQ:Segre}
\tilde{s}=\begin{cases} 
\mbox{ if } p=2k & ([\beta_1],[\beta_2])\in \bP(S^k)\times\bP(S^{k})\mapsto [\beta_1\cdot \beta_2]\in \bP(S^{p}) \\
\mbox{ if } p=2k+1 & ([\beta],[\gamma])\in \bP(S^k)\times\bP(S^{k+1})\mapsto [\beta\cdot \gamma]\in \bP(S^{p}).
\end{cases}
\end{equation}
We also denote by $D^{p}$ the image of $\tilde{s}$ and by $D_J^{p}$ the intersection $D^p\cap \bP(J^{p})$. In other words, elements of $D_J^p$ are (particular) decomposable elements which are also in the Jacobian ideal.
\vspace{2mm}

If $p$ is an integer such that $p\geq d-1$ we have that $S^{p-d}$ and $J^{p}$ are non-trivial (and $S^{p-d+1}=0$ if and only if $p=d$). Then we can carry on the following construction. For any $v=(v_0,v_1,v_2)\neq (0,0,0)$ let $F_v$ be the directional derivative $v\cdot \nabla F=v_0F_{x_0}+v_1F_{x_1}+v_2F_{x_2}$. If one denotes by $f_v$ the map
$$f_v:S^{p-d}\oplus S^{p-d+1}\to S^{p}\qquad f_v(A,B)=AF+BF_v$$
we have that $f_v$ is bilinear and injective. Indeed, having $(A,B)\neq (0,0)$ such that $AF+BF_v=0$ would yield $AF=-BF_v$ which is impossible, since $F$ is irreducible. 
Moreover, by construction, we have that $\Img(f_v)\subseteq J^{p}$. Let $\tilde{f_v}$ be the projectifization of $f_v$ and denote by 
$$E_v^{p}=\{[AF+BF_v]\,|\, A\in S^{p-d},\, B\in S^{p-d+1}\}$$ 
its image in $\bP(J^{p})$.
\vspace{2mm}

If $p\geq d$ we have then the diagram
$$
\xymatrix{
D^p\ar@{^{(}->}[r] & \bP(S^{p}) & & \\
 D_J^p\ar@{^{(}->}[u]\ar@{^{(}->}[r] & \bP(J^{p})\ar@{^{(}->}[u] & E_v^p\ar@{^{(}->}[l] & \bP(S^{p-d}\oplus S^{p-d+1})\ar@{->>}[l]_-{\tilde{f}_v}
}$$
where the vertical morphism is induced by the inclusion $J^{p}\subset S^{p}$.


\begin{proposition}
\label{PROP:JETSINJDEC}
Assume that $d-1 \leq p\leq 2d-4$ and take $v$ as above and general. 

Then $D_J^{p}\cap E_v^{p}$ is empty. In particular, 
\begin{equation}
\label{INEQ:BOUNDD_J}
\dim(D_J^p)=\dim(\bP(J^{p})\cap \Img(\tilde{s}))<\dim(\bP(J^p))-s_{p-d}-s_{p-d+1}+1
\end{equation}
where $\tilde{s}$ is the Segre morphism defined above.
\end{proposition}

\begin{proof}
Assume that $p=d-1$. Then $J^p=J^{d-1}$ is a linear system without base points and, by Bertini's Theorem, its general element is irreducible. On the other hand, if $p=d-1$ we have that $E_v^{p}=\{[F_v]\}$ is a single point so for $v$ general $E_v^{d-1}\cap D_J^{d-1}=\emptyset$. Since $s_{p-d}=0$ and $s_{p-d+1}=1$, Inequality \ref{INEQ:BOUNDD_J} holds.
\vspace{2mm}

Assume now that $p=2k\geq d$ so, by the hypothesis on $p$, we have 
\begin{equation}
\label{INEQ:In1}
k\leq d-2.
\end{equation}
Suppose, by contradiction, that there exists an element in $E^{2k}_v\cap D^{2k}_J$. Then, there exist $\beta_1,\beta_2\in S^{k}\setminus\{0\}$ and $A\in S^{2k-d},B\in S^{2k-d+1}$ such that
\begin{equation*}
\beta_1\beta_2=AF+BF_v.
\end{equation*}
The dual map of the curve $X$ fits into a diagram
$$
\xymatrix{
\bP^2 \ar[r]^-{\nabla F} & (\bP^2)^*\\
X \ar@{^{(}->}[u] \ar[r]_-{\nu} & X^*\ar@{^{(}->}[u]
}$$
where $\nabla F$ is the gradient of $F$. It is the morphism induced by the subsystem $|J^{d-1}|$ of $|S^{d-1}|$. It is well known that $\nu$ is a birational morphism and that $X^*$ is a curve of degree $d(d-1)$. 
Any choice of $v=(v_0,v_1,v_2)\neq (0,0,0)$ corresponds to the choice of a line in $(\bP^2)^*$, namely the line $L_v:v_0z_0+v_1z_1+v_2z_2=0$, if $z_i$ are projective coordinates on $(\bP^2)^*$. 
If $v$ is general, the corresponding line is also general so it cuts $X^*$ in exactly $d(d-1)$ distinct points. Since $v$ is general, we can also assume that the preimages of these points are distinct. Since $(\nabla F)^*(L_v)$ is the curve with equation $F_v=0$,
we have produced a set $P$ of $d(d-1)$ distinct points on $X$, which, by construction, also annihilate the polynomial $AF+BF_v$. 
In particular, we have that the product $\beta_1\beta_2$ vanish on all these points. 
Notice that neither $\beta_1$ nor $\beta_2$ can vanish on all the $d(d-1)$ points of $P$ since $\beta_i\in S^{k}=H^0(\cO_{\bP^2}(k))$ cuts on $X$ a divisor of degree $kd$ and by Inequality \eqref{INEQ:In1} we have $k\leq d-2$. Hence one of them (say $\beta_1$, for example) vanish on $m$ of theses points with 
\begin{equation}
\label{INEQ:In2}
m\geq d(d-1)/2
\end{equation}
and does not vanish on at least one of the other $d(d-1)-m$ points.
\vspace{2mm}

Let $\{p_1,\dots,p_m\}$ be the points of $P$ in the support of the divisor cut by $\beta_1$ on $X$ and let $P_{m+1}$ be another point chosen among the $d(d-1)$ points. By the Uniform Position Theorem, by following loops around $L_v$ we can permute the points $\{p_1,\dots,p_m,p_{m+1}\}$ via monodromy in any way we would like (because the monodromy is the full symmetric group of the fiber). Then we are able to construct $m+1$ sections $\omega_1,\dots,\omega_{m+1}$ of $\cO_{X}(k)$ such that $\omega_i(p_j)=\lambda_i\delta_{ij}$ where $\lambda_i\neq 0$ and $\omega_{m+1}=\beta_1|_{X}$. The above vanishing conditions ensure that these forms are also independent so we have $m+1\leq h^0(\cO_{X}(k))$. By Inequality \eqref{INEQ:In1} and since $s_k$ is an increasing sequence, we have $h^0(\cO_{X}(k))=h^0(\cO_{\bP^2}(k))=s_k\leq s_{d-2}$.
Then, by Inequality \eqref{INEQ:In2} we have the condition
$$\frac{d(d-1)}{2}+1\leq m+1\leq s_k\leq \frac{d(d-1)}{2}$$
which leads easily to a contradiction. The proof for $p\geq d$ with $p$ odd is analogous.
\vspace{2mm}

Since we have proved that $D_j^p$ and $E_v^{p}$ are disjoint we have also
$$\dim(E^{p}_v)+\dim(D^{p}_J)<\dim(\bP(J^{p})).$$
The last claim of the proposition, i.e. formula \eqref{INEQ:BOUNDD_J}, follows from the last inequality since, for $p\geq d$ we have $\dim(E_v^{p})=s_{p-d}+s_{p-d+1}-1$ by construction.
\end{proof}


The following corollary won't be used in what follows but, in our opinion is worth to be mentioned since it has relevant geometric meaning.


\begin{corollary}
\label{COR:EXPDIMENSION}
Assume that $d\geq 3$. Then, the intersection $D_J^{2d-4}=D^{2d-4}\cap \bP(J^{2d-4})$ has the expected dimension in $\bP(S^{2d-4})$. 
\end{corollary}

\begin{proof}
Since $d\geq 3$ we have $d-1\leq 2d-4$ and we can apply Proposition \ref{PROP:JETSINJDEC} with $p=2d-4$ and $k=d-2$. By Gorenstein duality we have $(R^{2d-4})^*\simeq R^{d-2}\simeq S^{d-2}$ so
$$\dim(\bP(J^{2d-4}))=s_{2d-4}-r_{2d-4}-1=s_{2d-4}-s_{d-2}-1=\frac{3d^2-9d+4}{2}.$$
On the other hand 
$$\dim(D^{2d-4})=2(s_{d-2}-1)=(d-2)(d+1)\quad\mbox{ and }\quad \dim(\bP(S^{2d-4}))=s_{2d-4}-1=(d-2)(2d-1)$$
so we have that the expected dimension $\Edim(D_J^{2d-4})$ of the intersection  $D_J^{2d-4}=D^{2d-4}\cap \bP(J^{2d-4})$ is
$$\Edim(D_J^{2d-4})=\dim(D^{2d-4})+\dim(\bP(J^{2d-4}))-\dim(\bP(S^{2d-4}))=\frac{(d-2)(d+1)}{2}-1.$$
Comparing this with the inequality given by Proposition \ref{PROP:JETSINJDEC} we get
$$\dim(D_{J}^{2d-4})\leq\frac{3d^2-9d+4}{2}-s_{d-4}-s_{d-3}+1-1=\Edim(D_J^{2d-4})$$
so $D_J^{2d-4}$ has actually the expected dimension.
\end{proof}


The following proposition gives a criterion which we will use to prove the main theorem.

\begin{proposition}
\label{PROP:IMAXINEQ}
Let $X$ be a smooth planar curve of degree $d\geq 5$ and genus $g$. If $\dim(D_J^{2d-6})< g-1$ we have that $X$ has $I$-maximal variation as plane curve.
\end{proposition}

\begin{proof}
Assume that $X$ is not of $I$-maximal variation. We will prove that $\dim(D_J^{2d-6})\geq g-1$.
Since $d\geq 5$ we have $2d-6\geq d-1$ and then $J^{2d-6}$ is not empty. Recall that, according to Equation \eqref{EQ:Segre}, for $p=2d-6$, we have the Segre morphism $\tilde{s}:\bP(S^{d-3})\times \bP(S^{d-3})\to \bP(S^{2d-6})$ such that $([\alpha],[\beta])\mapsto [\alpha\beta]$ and $D^{2d-6}_J$ is, by definition, the intersection of $\bP(J^{2d-6})$ with the image of $\tilde{s}$.
\vspace{2mm}

In order to prove the desired inequality, set $$Y=\{[\alpha]\in\bP(R^{d-3})\,|\, \alpha^2\in J^{2d-6}\}$$ and consider the incidence correspondence $$Z=\{([\alpha],[\beta])\in Y\times \bP(R^{d-3})\,|\, \alpha\beta\in J^{2d-6}\}$$
with its projection $\pi_1$ and $\pi_2$. Then we have a diagram
$$
\xymatrix{
\bP(J^{2d-6}) & Z=\{([\alpha],[\beta])\in Y\times \bP(R^{d-3})\,|\, \alpha\beta\in J^{2d-6}\}\ar[r]^-{\pi_2}\ar[d]_-{\pi_1}\ar[l]_-{\psi}\ar[ld]_-{\psi} & \bP(R^{d-3}) \\
D_J^{2d-6}\ar@{^{(}->}[u] & Y=\{[\alpha]\in\bP(R^{d-3})\,|\, \alpha^2\in J^{2d-6}\}
}$$
where $\psi:Z\to \bP(J^{2d-6})$ is the multiplication morphism $\psi([\alpha],[\beta])=[\alpha\beta]$. Note that it has image in $D^{2d-6}_J$ by construction and it is finite since it is the restriction of the Segre morphism $\tilde{s}$. 
Then, we have $\dim(D^{2d-6}_J)\geq \dim(\Img(\psi))=\dim Z$. By construction $\pi_1$ is surjective (indeed, if $[\alpha]\in Y$, we have $([\alpha],[\alpha])\in Z$) and $\pi_1^{-1}([\alpha])=[\alpha]\times \bP(K_{d-3}(\alpha))$ so, for $\alpha$ general we have
$$\dim(D^{2d-6}_J)\geq \dim(Z)=\dim(Y)+k_{d-3}(\alpha)-1.$$

In particular, in order to conclude the proof, it is enough to prove that 
\begin{equation}
\label{INEQ:EQASS1}
\dim(Y)+k_{d-3}(\alpha)\geq g.
\end{equation}

Consider the incidence correspondence 
$$\tilde{\cI}=\{([\eta],[\alpha])\in \bP(R^{d})\times\bP(R^{d-3})\,|\, \eta\alpha\in J^{2d-3}\}$$
with its projections $p_1$ and $p_2$. Since $X$ is not of $I$-maximal variation we have that $p_1$ is surjective. Then, there is an irreducible component $\cI$ of $\tilde{\cI}$ which dominates $\bP(R^{d})$ via $p_1$. We will denote by $p_i$ also the restriction of $p_i$ to $\cI$ for brevity. Since $p_1$ is dominant we have
\begin{equation}
\label{INEQ:1}
\dim(\cI)\geq \dim \bP(R^d)=r_d-1.
\end{equation}
Let $U$ be the open dense subset of $\cI$ with the pairs $([\eta],[\alpha])$ with $\eta\alpha\in J$ and $\eta\cdot:R^{d-3}\to R^{2d-3}$ of maximal rank. Then, by Lemma \ref{LEM:kercokerLAST}, if $([\eta],[\alpha])\in U$ we have $\alpha^2\in J^{2d-6}$. Hence we have $p_2|_U:U\to Y$ and $\dim(p_2(U))=\dim(p_2(\cI))\leq \dim(Y)$.
Then we can write $$\dim(p_2^{-1}([\alpha]))\geq \dim(\cI)-\dim(\im(p_2)).$$
By Inequality \eqref{INEQ:1} and since the fiber over $[\alpha]\in p_2(\cI)$ is $p_2^{-1}([\alpha])=\bP(K_{d}(\alpha))\times [\alpha]$, we have
\begin{equation}
r_d-1\leq \dim(\cI)\leq k_{d}(\alpha)-1+\dim(Y).
\end{equation}
In particular, by Lemma \ref{gring}, we have that
\begin{equation}
\label{EQ:2}
\dim(Y)\geq r_d-k_d(\alpha)=r_{d-3}-k_{d-3}(\alpha)=g-k_{d-3}(\alpha).
\end{equation}
which yields Inequality \eqref{INEQ:EQASS1}, as claimed.
\end{proof}


We are now ready to prove our main theorem.

\begin{proof}
The Theorem is true for $d=3$. Indeed, let $\pi:\cX\to B$ be the family of smooth planar curves of degree $3$. Then, for all $b\in B$ we have that $\KS_b:T_{B,b}\to H^1(T_{X_b})\simeq \bC$ is not zero and so it is surjective. Then, by Lemma \ref{LEM:FULLMAXVAR} we have that $X_b$ has $I$-maximal variation as planar curve. 
\vspace{2mm}

Assume now $d\geq 4$ and proceed by contradiction by assuming also that $X$ is not of $I$-maximal variation. 
\vspace{2mm}

If $d=4$, for $\eta\in R^4$ general we have that $\eta\cdot: R^1\to R^5$ has non trivial kernel. If $\alpha\in K_{1}(\eta)$, by Lemma \ref{LEM:kercokerLAST}, we would have $\alpha^2\in J^{2}$. Since $J^{2}=\{0\}$ (as $J$ is generated in degree $d-1=3$) we have $\alpha=0$ which gives a contradiction. One can also prove the thesis in this case using that quartic curves are canonical.
\vspace{2mm}

Let $d\geq 5$. Under this assumption we can apply Proposition \ref{PROP:IMAXINEQ}. By doing so we have the inequality
\begin{equation}
\label{INEQ:ASS1}
\frac{(d-1)(d-2)}{2}-1=g-1\leq \dim(D^{2d-6}_J).
\end{equation}
Notice that, if we assume $d=5$, this yields $\dim(D_J^{4})\geq 5$ which is impossible since $\dim(D_J^{4})\leq \dim(\bP(J^4))=2$. Hence we can assume that $d\geq 6$.
\vspace{2mm}

Inequality \eqref{INEQ:ASS1} gives us a bound for the dimension of $D_J^{2d-6}$ from below. We aim now to get a bound from above. Since $d\geq 6$, if we set $k=d-3$ and $p=2k$ we have $d\leq p\leq 2d-4$ so we can apply Proposition \ref{PROP:JETSINJDEC} to obtain
$$\dim(D_J^{2d-6})<\dim(\bP(J^{2d-6}))-s_{d-6}-s_{d-5}+1.$$
Since $(R^{2d-6})^*\simeq R^{d}$, by Gorenstein duality, we have
$$\dim(J^{2d-6})=s_{2d-6}-r_{2d-6}=s_{2d-6}-r_{d}=s_{2d-6}-(s_{d}-9)$$
so the above inequality yields
\begin{equation}
\label{INEQ:ASS2}
\dim(D_J^{2d-6})<s_{2d-6}-s_d-s_{d-6}-s_{d-5}+9=\frac{(d-1)(d-4)}{2}
\end{equation}

We can conclude the Theorem by observing that Inequalities \eqref{INEQ:ASS1} and \eqref{INEQ:ASS2} lead to a contradiction. Indeed, from the two inequalities we have
$$\frac{d(d-3)}{2}=\frac{(d-1)(d-2)}{2}-1=g-1\leq \dim(\bP(J^{2d-6}))<\frac{(d-1)(d-4)}{2}$$
so we have $d(d-3)<(d-1)(d-4)$ which is true if and only if $d<2$. Then planar curves of degree $d\geq 3$ have $I$-maximal variation.

\end{proof}


\section{Yukawa Coupling for hypersurfaces}
\label{SEC:YC}

Let $n\geq 2$ and consider $S=\bK[x_0,\dots,x_n]$, the homogenous coordinate ring of $\PP^n_{\bK}$ with the standard graduation $S=\bigoplus_{m\geq 0} S^m$. Consider the Fermat polynomial of degree $d$, i.e. $F_d=\sum_{i=0}^nx_i^d$,
and its Jacobian ideal $$J_{F_d}=J_d=(x_0^{d-1},\dots x_n^{d-1}).$$ 
We set 
$$J_d^k=J_d\cap S^k,\qquad H=F_1=\sum_{i=1}^nx_i \quad \mbox{ and }\quad \sigma_d=\left(\Pi_{i=0}^{n}x_i\right)^{d-2}$$ so that
$\sigma_d$ is a generator for the socle of the Jacobian ring $R_d=S/J_d$.

For $d,k\geq 0$ consider the following property:
$$(\star)_{d,k}: \qquad \mbox{ if } G\in S^k, G\neq 0\quad \Longrightarrow \quad G\cdot H^{d(n-1)}\not\in J_{d} \, (\mbox{more precisely, } J_d^{k+d(n-1)}).$$

\begin{lemma}
\label{LEM:MUNSTER}
Property $(\star)_{d,k}$ is true if $d\geq n+1$ and $k\leq d-n-1$.
\end{lemma}

\begin{proof}
First of all, notice that it is enough to prove the statement for $k=d-n-1$. Indeed, if $(\star)_{d,d-n-1}$ holds and if $G'\in S^{k}$ with $G'\neq 0$ and $G'\cdot H^{d(n-1)}\in J_d^{k+d(n-1)}$ for some $k<d-n-1$ then we would have
$G'\cdot H^{d(n-1)}\cdot H^{e} \in J_d^{k+d(n-1)+e}$ for all $e\geq 1$. In particular, for $e=d-n-1-k$ we would have 
$$G'\cdot H^{d(n-1)}\cdot H^{d-n-1-k}=(G'\cdot H^{d-n-1-k})\cdot H^{d(n-1)}\in J_{d}$$ but this is impossible as $G'\cdot H^{d-n-1-k}\in S^{d-n-1}\setminus \{0\}$
and we are assuming $(\star)_{d,d-n-1}$. Hence $(\star)_{d,k}$ holds also for $k<d-n-1$. 
\vspace{2mm}

Set $k=d-n-1$. We will prove now that $(\star)_{d,k}$ holds by induction on $d\geq n+1$. First of all notice that if $d=n+1$ we have $k=0$ and the claim is equivalent to say that $H^{d(n-1)}\not \in J_d$. Since $d=n+1$ we have $$H^{d(n-1)}=H^{(n+1)(d-2)}=\cdots + \frac{[d(d-2)]!}{(d-2)!d}\left(\Pi_{i=0}^nx_i\right)^{d-2}+\cdots$$
so $H^{d(n-1)}$ is equal to $\lambda\cdot \sigma_d$ in $R_d$ with $\lambda\neq 0$. Since $\sigma_d$ generates the socle of $R_d$ we have that $H^{d(n-1)}\not\in J_d$, as claimed.
\vspace{2mm}

Assume now, as induction hypothesis, that $(\star)_{l,l-n-1}$ is true for $l\leq d-1$. Let $G\in S^{d-n-1}$ and assume $L=G\cdot H^{d(n-1)}\in J_d^{dn-n-1}$. In order to conclude our proof we need to show that $G=0$.
\vspace{2mm}

For simplicity we will introduce the following notations. Set $I_n=\{0,1,\dots, n\}$. If $I\subseteq I_n$ we set $|I|$ to be the cardinality of $I$ and call it length of $I$. 
For any polynomial $g\in S^m$ and $I=\{i_1,\dots,i_r\}\subseteq I_n$ denote by $\p_I(g)$ the derivative of $g$ with respect to all the variables $\{x_{i_1},\dots,x_{i_r}\}$. 
If $I\subseteq$ we set $\hat{I}=I_n\setminus I$ and $\ph_I:=\p_{\hat{I}}$, i.e. the derivative with respect to all the variables with indices not in $I$. For brevity, we will write $\p_i$ and $\ph_i$ instead of $\p_{\{i\}}$ and $\ph_{\{i\}}$. 
\vspace{2mm}

Since $L= G\cdot H^{d(n-1)}\in J^{n(d-1)-1}_d=(x_0^{d-1},\dots,x_n^{d-1})$ we can write it as
\begin{equation}
\label{EQ:GHL}
L=\sum_{i=0}^{n}x_i^{d-1}f_i
\end{equation}
for suitable $f_i\in S^{n(d-1)-n}$. We claim that $\ph_{j}(L)\in J_{d-1}$ for all $j$. This follows easily since $J_{d-1}$ is generated by the monomials $x_0^{d-2},\dots,x_n^{d-2}$ and since
\begin{multline*}
\ph_{j}(L)=\sum_{i=0}^n\ph_j(f_ix_i^{d-1})=\ph_j(f_jx_j^{d-1})+\sum_{i\neq j}\ph_j(f_ix_i^{d-1})=x_j^{d-1}\ph_j(f_j)+
\sum_{i\neq j}\ph_{\{i,j\}}(\p_i(f_ix_i^{d-1}))=\\
=x_j^{d-1}\ph_j(f_j) + \sum_{i\neq j}\ph_{\{i,j\}}(\p_i(f_i)x_i^{d-1}+(d-1)f_ix_i^{d-2})=
\sum_{i=0}^{n}\ph_{\{j\}}(f_i)x_i^{d-1}+(d-1)\sum_{i=0}^{n}f_ix_i^{d-2}.
\end{multline*}

Now we want to express the difference $\ph_i(L)-\ph_j(L)$ (which, as we have just shown, is an element of $J_{d-1}$). In order to do so, notice that, for $I\subseteq \{0,\dots,n\}$ and $s\geq |I|$ we have
\begin{equation}
\label{EQ:DERGHs}
\p_I(GH^{s})=\sum_{m=0}^{|I|}\frac{s!}{(s-|I|+m)!}H^{s-|I|+m}\left(\sum_{\substack{|J|=m\\J\subseteq I}}\p_J\right)(G).
\end{equation}
In particular, using Equation \eqref{EQ:DERGHs}, we can write 
\begin{equation}
\label{EQ:DERGHs2}
\ph_i(GH^{d(n-1)})=\lambda_{0}GH^{d(n-1)-n} + H^{(d-1)(n-1)}\sum_{m=1}^{n}\lambda_{m}H^{m-1}\left(\sum_{\substack{|J|=m\\i\not\in J}}\p_J\right)(G)
\end{equation}
where all the coefficients $\lambda_m$ are strictly positive for $m=0,\dots,n$ and $\lambda_n=1$. Hence, if $i\neq j$, we have

$$J_{d-1}\ni (\ph_i-\ph_j)(GH^{d(n-1)})= H^{(d-1)(n-1)}\left[\sum_{m=1}^{n}\lambda_{m}H^{m-1}\left(\sum_{\substack{|I|=m\\i\not\in I}}\p_I-\sum_{\substack{|J|=m\\j\not\in J}}\p_J\right)(G)\right].$$

Then, the sum in the square bracket is $0$ by the induction hypothesis $(\star)_{d-1,(d-1)-n-1}$ since it is an element of $S^{(d-1)-n-1}$ which multiplied by $H^{(d-1)(n-1)}$ is in $J_{(d-1)}$.

$$\sum_{m=1}^{n}\lambda_{m}H^{m-1}\Delta^{m}_{i,j}(G)=0\qquad \mbox{ with }\qquad \Delta^{m}_{i,j}=\left(\sum_{\substack{|J|=m\\i\not\in J}}\p_J-\sum_{\substack{|J|=m\\j\not\in J}}\p_J\right).$$

If $|J|=m$ and $i,j\in J$ then $|J|$ does not give contribution to the above sum. On the other hand the same is true if both $i$ and $j$ are not in $J$ since the contributions cancels out. If $J$ give a contribution to the sum then either $i\not\in J$ and $J=J'\cup \{j\}$ or $j\not\in J$ and $J=J'\cup\{j\}$. Then we can write
\begin{equation}
\Delta^m_{i,j}=\sum_{\substack{|J|=m\\i\not\in J,j\in J}}\p_J-\sum_{\substack{|J|=m\\j\not\in J,i \in J}}\p_J=
\sum_{\substack{|J|=m-1\\i,j\not\in J}}\p_J\p_j-\sum_{\substack{|J|=m-1\\i,j\not\in J}}\p_J\p_i=\left(\sum_{\substack{|J|=m-1\\i,j\not\in J}}\p_J\right)(\p_j-\p_i)
\end{equation}
and we have proven that for all $i\neq j$, if we set $G_{i,j}=(\p_j-\p_i)(G)$, $G_{i,j}'$ satisfy the following differential equation:
\begin{equation}
\label{EQ:DIFFEQ}
\sum_{m=1}^{n}\lambda_{m}H^{m-1}\Gamma_{i,j}^{m-1}(G_{i,j})=0 \qquad \mbox{ where }\qquad \Gamma_{i,j}^{m-1}=\sum_{\substack{|J|=m-1\\i,j\not\in J}}\p_J
\end{equation}
with the coefficients $\lambda_m$ strictly positive and $\Gamma_{i,j}^{0}=\id$. 
\vspace{2mm}

We claim that $G_{i,j}=0$. Notice that $\lambda_1\neq 0$ and $\Gamma_{i,j}^{0}=\id$ imply that $G_{i,j}$ is divisible by $H$ so we can write $G_{i,j}=HG_{i,j}'$. We claim that $G_{i,j}'$ satisfies a differential equation like the one in \eqref{EQ:DIFFEQ} (with different coefficients $\lambda_m$ but always strictly positive). If $I=\{i_1,\dots,i_m\}$ then
$$\p_I(H K) = \sum_{\substack{k=1}}^{m}\p_{I\setminus\{i_k\}}(K)+H\p_I(K)\quad \mbox{ so }\quad \Gamma_{i,j}^{m-1}(H\cdot K)=\Gamma_{i,j}^{m-2}(K)+H\Gamma_{i,j}^{m-1}(K).$$
Hence, from Equation \eqref{EQ:DIFFEQ} we have
\begin{multline}
0=\sum_{m=1}^{n}\lambda_{m}H^{m-1} \Gamma_{i,j}^{m-1}(HG_{i,j}')=\lambda_1HG_{i,j}'+\sum_{m=2}^{n}\lambda_{m}H^{m-1} (\Gamma_{i,j}^{m-2}(G_{i,j}')+H\Gamma_{i,j}^{m-1}(G_{i,j}'))=\\
=\lambda_1HG_{i,j}'+\sum_{m=2}^{n}\lambda_{m}H^{m-1} \Gamma_{i,j}^{m-2}(G_{i,j}')+\sum_{m=2}^{n}\lambda_{m}H^{m}\Gamma_{i,j}^{m-1}(G_{i,j}'))=\\
=\lambda_1HG_{i,j}'+\lambda_2HG_{i,j}'+\sum_{m=2}^{n-1}(\lambda_{m}+\lambda_{m+1})H^{m} \Gamma_{i,j}^{m-1}(G_{i,j}')+\lambda_{n}H^{n}\Gamma_{i,j}^{n-1}(G_{i,j}').
\end{multline}
By dividing by $H$ we get 
$$\sum_{m=1}^{n}\lambda_m'H^{m-1}\Gamma_{i,j}^{m-1}(G_{i,j}')=0 \qquad \mbox{ where }\qquad \lambda_m'=\lambda_{m}+\lambda_{m+1} \mbox{ for }m\leq n-1\mbox { and }\lambda_{n}'=\lambda_n.
$$
Since $\lambda_m>0$ we have, as claimed, that the coefficients of the differential equation are strictly positive. In particular, $\lambda_1'\neq 0$ and so we obtain, as before, that $G_{i,j}'$ is divisible by $H$ and we can iterate this process. After a finite number of iteration of this process we have $G_{i,j}=H^{d-n-3}\cdot G_{i,j}''$ with $G_{i,j}''\in S^1$ that satisfy an equation like \eqref{EQ:DIFFEQ} with coefficients $\lambda_m''>0$ for all $m$. Since $G_{i,j}''\in S^1$ we have $\Gamma_{i,j}^{m-1}(G_{i,j}'')=0$ as soon as $m\geq2$. Then the differential equation satisfied by $G_{i,j}''$ is simply $\lambda_1''G_{i,j}''=0$ which yields $G_{i,j}''=0$ and then, finally $G_{i,j}=H^{d-n-3}\cdot 0=0$ as claimed.
\vspace{2mm}

Since $G_{i,j}=(\p_i-\p_j)(G)=0$ for all $i,j$ we have $\p_i(G)=\p_j(G)$ for all $i,j$. We claim that $K\in S^m$ with $\p_0(K)=\cdots=\p_n(K)$ can be written as $\alpha H^{m}$ for suitable $\alpha$. If $m=1$ this is clear. If we assume that the claim holds till $m-1$ and $K\in S^m$ with $\p_0(K)=\cdots=\p_n(K)=D$ we have that $D\in S^{m-1}$ and satisfies the same condition. Indeed, if $i,j$ are two different indices, we have $\p_i(D)=\p_i(\p_j K)=\p_j(\p_i K)=\p_j(D)$. Then, by induction we have $D=\alpha'H^{m-1}$. Using Euler relation we have
$$K=\frac{1}{m}\sum_{i=0}^nx_i\p_iK=\frac{1}{m}\sum_{i=0}^nx_i\alpha'H^{m-1}=\frac{1}{m}\alpha'H^m,$$
as claimed and we get $G=\alpha H^{d-n-1}$.
\vspace{2mm}

Since $G=\alpha H^{d-n-1}$ and $G\cdot H^{d(n-1)}\in J_d$ we have that $\alpha H^{d(n-1)+d-n-1}$ is $0$ in the Jacobian ring $R_d$. Hence, $\alpha H^{d(n-1)+d-n-1}\cdot H^{d-n-1}=\alpha H^{(n+1)(d-2)}$ is also $0$ in the Jacobian ring. As $H^{(d-2)(n-1)}=\lambda\cdot \sigma_d$ in $R$ with $\lambda\neq 0$ as observed ad the beginning of the proof of this theorem we have $0=\alpha\lambda\sigma_d$ in $R$ which is only possible if and only if $\alpha=0$, i.e. if and only if $G=0$.
\end{proof}

\begin{proposition}
\label{PROP:MAXVARHYPERSUP}
The general hypersurface in $Y\subseteq \bP^n$ of degree $d\geq n+1$ has $I$-maximal variation, i.e. $\delta_M(\cO_{\bP^n}(d))=h^0(\omega_Y)=h^0(\cO_{\bP^{n}}(d-n-1))$.
\end{proposition}

\begin{proof}
Let $Y$ be a Fermat hypersurface in $\bP^n$ of degree $d\geq n+1$. Then the Yukawa coupling associated to $Y$ is generically an isomorphism by Lemma \ref{LEM:MUNSTER}. By semicontinuity, this holds also for the general hypersurface in $\bP^n$ of degree $d\geq n+1$. In particular, $\delta_M(\cO_{\bP^n}(d))$ is maximal.

\end{proof}

As already said in the introduction, the above proposition can be also seen as a consequence of the Theorem recalled in Remark \ref{REM:SUNTOLEF} and the discussion in the Remark \ref{REM:SLP} and Lemma \ref{LEM:SLPMAXVAR}.

\end{document}